\documentclass{article} 
\usepackage[utf8]{inputenc}
\usepackage[a4paper]{geometry}
\usepackage[english]{babel}
\usepackage{color}

\usepackage{fullpage}

\usepackage{amssymb,amsmath,mathtools,amsthm,dsfont}

\usepackage[all]{xy}

\usepackage[usenames, dvipsnames]{xcolor}
\usepackage{tikz}
\usetikzlibrary{patterns}
\usepackage[justification=centering]{caption}
\usepackage{pgfplots}
\pgfplotsset{compat=1.9}
\usepackage{mathrsfs}
\usetikzlibrary{arrows}

\newtheorem{theorem}{Theorem}
\newtheorem{prop}[theorem]{Proposition}
\newtheorem{corollary}[theorem]{Corollary}
\newtheorem{lemma}[theorem]{Lemma}
\newtheorem{claim}[theorem]{Claim}

\newcommand{\E}{\mathbb{E}}

\newcommand{\N}{\mathbb{N}}

\renewcommand{\P}{\mathbb{P}}

\newcommand{\R}{\mathbb{R}}

\newcommand{\Z}{\mathbb{Z}}

\newcommand{\cC}{\mathcal{C}}

\newcommand{\cE}{\mathcal{E}}
\newcommand{\cF}{\mathcal{F}}

\newcommand{\cT}{\mathcal{T}}

\newcommand{\cX}{\mathcal{X}}

\renewcommand{\1}{\mathds{1}}

\renewcommand{\epsilon}{\varepsilon}
\renewcommand{\phi}{\varphi}

\newcommand{\diam}{\text{diam}}

\begin{document}

\selectlanguage{english}

\title{Disjoint finite geodesics in first-passage percolation}

\author{Olivier Durieu\footnote{olivier.durieu@univ-tours.fr}, Jean-Baptiste Gou\'er\'e\footnote{jean-baptiste.gouere@univ-tours.fr}, Antonin Jacquet\footnote{antonin.jacquet@univ-tours.fr}\\
\small{Institut Denis Poisson, UMR-CNRS 7013, Université de Tours}}
\maketitle

\begin{abstract}
We investigate first-passage percolation on the lattice $\Z^d$ for dimensions $d \geq 2$.
Each edge $e$ of the graph is assigned an independent copy of a non-negative random variable $\tau$.
We only assume $\P[\tau=0]<p_c(\Z^d)$, the critical probability threshold for Bernoulli bond percolation on $\Z^d$. 
We obtain lower bounds of order $n^{-\alpha_d}$ (where $\alpha_d>0$ is explicit) for the probability of having two disjoint geodesics between two pairs of neighbouring vertices at distance $n$.
Additionally, under more specific assumptions on the distribution of $\tau$, we obtain similar lower bounds for the probability of having two disjoint geodesics (except for their starting and ending points) between the same two vertices.
\end{abstract}

\section{Introduction}

\subsection{First-passage percolation} 
Let $d \ge 2$. 
We consider on $\Z^d$ the usual undirected graph structure.
In other words two vertices $x,y \in \Z^d$ are linked by an edge if the Euclidean distance between them is one.
We write $x \sim y$ when this is the case.
We denote by $\cE$ the set of edges.
A path is a finite sequence $\pi=(x_0,\dots,x_n)$ of vertices of $\Z^d$ such that, for all $i \in \{0,\dots,n-1\}$, $x_i$ and $x_{i+1}$ are neighbors, that is,
they are linked by an edge.
We say that this is a path from $x$ to $y$ if its first element is $x$ and its last element is $y$.

Let $\tau$ be a random variable with values in $[0,+\infty)$.
Let $(\tau(e))_{e \in \cE}$ be a family of independent copies of $\tau$.
If $\pi=(x_0,\dots,x_n)$ is a path, we set
\[
\tau(\pi) = \sum_{i=0}^{n-1} \tau(x_i,x_{i+1}).
\]
For any vertices $x,y \in \Z^d$ we then define $T(x,y)$ by
\begin{equation} \label{e:definiton_T}
T(x,y) = \inf_{\pi : x \to y} \tau(\pi)
\end{equation}
where the infimum is taken on all paths from $x$ to $y$.
This defines a pseudometric on $\Z^d$.
This is a metric when $\tau$ only takes positive values.

This model has been introduced by Hammersley and Welsh in \cite{Hammersley-Welsh}.
Classical references on first-passage percolation include the Saint-Flour lecture notes by Kesten \cite{Kesten-St-Flour} and the monograph by Auffinger, Damron and Hanson 
\cite{50}.

Let $p_c(\Z^d)$ be the critical threshold for Bernoulli bond percolation on $\Z^d$.
We refer to the book by Grimmett \cite{Grimmett} for background on percolation.
The behaviour of the random pseudometric $T$ depends crucially on whether $\P[\tau=0]$ is small, equal or larger than $p_c(\Z^d)$.
In the whole of this article we work under the following assumption:
\begin{equation}\label{e:running_assumption}
\P[\tau=0] < p_c(\Z^d).
\end{equation}

A path $\gamma$ between the vertices $x$ and $y$ is called a geodesic if the infimum in \eqref{e:definiton_T} is achieved with $\gamma$, that is if $T(x,y)=\tau(\gamma)$.
Under \eqref{e:running_assumption}, with probability one, there exist geodesics between any pair of vertices.
This is Proposition 4.4 in \cite{50}.
Note however that our assumption does not imply uniqueness.
For any vertices $x,y \in \Z^d$ we denote by $\Gamma(x,y)$ the set of all geodesics from $x$ to $y$:
\[
\Gamma(x,y)=\{\text{geodesics between } x \text{ and } y\}.
\]

\subsection{Further notations.}

For all $n \ge 1$ we set
\[
S_n = \{x \in \Z^d : \|x\|_\infty = n\}
\]
where $\|\cdot\|_\infty$ is the usual supremum norm.
When $\pi$ is a path we denote by $|\pi|_e$ the number of edges of the path. 
When $\pi$ and $\pi'$ are paths, we denote by $\pi \cap \pi'$ the set of vertices common to both paths.
For any set $A \subset \Z^d$ we denote by $|A|$ its cardinality and by $\partial A$ its inner boundary:
\[
\partial A = \{x \in A : \text{ there exists } y \in \Z^d \setminus A \text{ such that } x \sim y\}.
\]
For any set $B\in\R^d$ we denote by $\diam(B)$ is diameter with respect to the $\|\cdot\|_\infty$ norm.

\subsection{Main results}

Our first main result is the following theorem.
It provides some lower bound on the probability of existence of two disjoint geodesics between two pairs of neighboring vertices. Below, $e_1$ and $e_2$ stand for the two first unit vectors of the canonical basis of $\Z^d$.

\begin{theorem}\label{t:main-neighbor} Assume \eqref{e:running_assumption}. There exists $C=C(d,\,\text{law of }\tau)>0$ such that the following holds.
\begin{enumerate}
\item For all $n \ge 1$,
\[
\P\big[\text{there exists $\gamma$ in $\Gamma(0,ne_1)$ and $\gamma'$ in $\Gamma(e_2,ne_1+e_2)$ such that $\gamma \cap \gamma'=\emptyset$}\big] \ge \frac C {n^{2d}}.
\]
\item For all $n \ge 1$, there exists $y,y' \in S_n$ such that $y\sim y'$ and 
\[
\P\big[\text{there exists $\gamma$ in $\Gamma(0,y)$ and $\gamma'$ in $\Gamma(e_2,y')$ such that $\gamma \cap \gamma'=\emptyset$}\big] \ge \frac C {n^{d-1}}.
\]
\end{enumerate}
\end{theorem}

Here is our second main result.
It provides some lower bound on the probability of existence of two disjoint geodesics between two vertices.

\begin{theorem}\label{t:main-same} Assume \eqref{e:running_assumption} and the existence of an integer $S \ge 1$ such that
\begin{equation}\label{e:support-mu-0S}
\text{the support of the law of } \tau  \text{ is the set } \{0,1,\dots,S\}.
\end{equation}
There exists $C=C(d,\,\text{law of }\tau)>0$ such that the following holds.
\begin{enumerate}
\item For all $n \ge 1$,
\[
\P\big[\text{there exists } \gamma \text{ and } \gamma' \text{ in } \Gamma(0,ne_1) \text{ such that } \gamma \cap \gamma' = \{0,ne_1\}\big] \ge \frac C {n^{2d}}.
\]
\item For all $n \ge 1$ there exists $u \in S_n$ such that
\[
\P\big[\text{there exists } \gamma \text{ and } \gamma' \text{ in } \Gamma(0,u) \text{ such that } \gamma \cap \gamma' = \{0,u\}\big] \ge \frac C {n^{d-1}}.
\]
\end{enumerate}
\end{theorem}

\subsection{Comments}

\paragraph{Upper bounds.} In \cite{Dembin-Elboim-Peled}, Dembin, Elboim, and Peled prove a result providing, as a particular consequence, an upper bound for the probability of the existence of disjoint geodesics. 
They work in the two dimensional setting under the following assumptions: 
the law of $\tau$ is absolutely continuous and admits some exponential moments, the asymptotic shape possesses more than $32$ sides.
As the law of $\tau$ is absolutely continuous, for any pair of vertices $u,v$ there exists a unique geodesic which we denote by $\gamma(u,v)$.
Their result implies the existence of $C>0$ and $\alpha>0$ such that
\[
\P\big[\text{there exists $\gamma$ in $\Gamma(0,ne_1)$ and $\gamma'$ in $\Gamma(e_2,ne_1+e_2)$ such that $\gamma$ and $\gamma'$ are disjoint}\big] \le \frac C {n^\alpha}.
\]
Their result is actually much stronger as it gives (in a quantitative way) an upper bound for the probability of existence of non coalescing geodesics starting near
some point and ending near another point.
We refer to \cite[Theorem 1.1]{Dembin-Elboim-Peled} for more details.

\paragraph{Comments on the results of Theorem 2.} Assume that there exists two disjoint geodesics $\gamma$ and $\gamma'$ between two points $x$ and $y$.
If we lower the passage time of any edge on $\gamma$, then any geodesic must use this edge.
This suggests that having two disjoint geodesics is very fragile, and it may suggest that the probability of such an event is sub-exponential in $\|x-y\|$.
Theorem 2 shows that this is not the case.

\paragraph{Comments on the proofs.}
The proofs of $2d$ exponent results rely on some simple averaging and symmetrization tricks and on a modification argument.
As an intermediate step (before symmetrization) we get similar results for point to hyperplane geodesics (see Propositions \ref{p:disjoint-half-space} and \ref{p:disjoint-half-space-same}).
Note that Item 1 of Theorem \ref{t:main-same} can not be deduced (as shown by counter-examples) from Item 1 of Theorem \ref{t:main-neighbor} by modifying the passage times of the edges $\{0,e_2\}$ and $\{ne_1,ne_1+e_2\}$.
In order to get Item 1 of Theorem \ref{t:main-same}, we therefore perform the modification argument on the intermediate step and then symmetrize.

Item 2 of Theorem \ref{t:main-neighbor} is (at least under some additional mild assumptions) a simple consequence of coexistence results.
We give more details in Section \ref{s:coex-intro}.
Item 2 of Theorem \ref{t:main-same} uses a modification argument.
As above and for the same reason, we do not perform the modification argument starting from the event of Item 2 of Theorem \ref{t:main-neighbor}.
Instead we start from a more restricted event where we have information on the four passage times $T(0,y), T(0,y'), T(e_2,y)$ and $T(e_2,y')$.

\paragraph{Comments on $2d$ exponent results \textit{vs} $d-1$ exponent results.}
\begin{itemize}
\item Obviously, $d-1$ exponent bounds are better than $2d$ exponent bounds.
\item The $2d$ exponent results are obtained with a simple proof and provide a result for a given pair of endpoints ($ne_1$ or $ne_1+e_2$) or a given endpoint ($ne_1$).
However, the proof relies on a symmetrization trick which prevents giving a result for arbitrary given pair of endpoints or arbitrary single endpoints.
Moreover, this symmetrization makes heavy use of the i.i.d.\ setting.
\item The $d-1$ exponent results relies on coexistence results in competition models. 
The same proof would work in a stationary and ergodic context with appropriate integrability assumptions (for both theorems) and appropriate finite energy assumptions (for the 
second theorem).
In this article we stick to the i.i.d.\ framework for simplicity.
\end{itemize}

\paragraph{Comments on the assumptions of Theorem 2.} Some assumptions are needed for this kind of results.
For example, if the law of $\tau$ is continuous, then there exists a unique geodesics between any two given points.
The assumption \eqref{e:support-mu-0S} is not optimal but it already encompass a rather large class of distributions.
Some weaker assumptions are possible, at the expense of lengthening the proof and/or worsening the exponent.
We refrain from developing these alternate assumptions in the article in order to keep it short and non technical.
However, we do not know what the optimal assumptions would be.
We emphasize it below as an open question.

\paragraph{Open question.} Under which condition on the law of $\tau$ does there exist an exponent $\alpha>0$ such that the probabilities studied in Theorem \ref{t:main-same} 
are of order at least $\frac 1 {n^{\alpha}}$ ?

\paragraph{Related results.} 
The closest results we are aware of is the above mentioned upper bound provided by \cite{Dembin-Elboim-Peled}
and the above mentioned results on coexistence.
Let us very quickly review some recent results on (mostly semi-infinite) geodesics.
The recent works by Ahlberg and Hoffman \cite{Ahlberg-Hoffman} and Ahlberg, Hanson and Hoffman \cite{Ahlberg-Hanson-Hoffman} 
provide a very detailed qualitative picture of semi-infinite geodesics in planar first-passage percolation under very mild assumptions. 
We highlight two results among the many established.
In \cite{Ahlberg-Hoffman}, they prove that the probability of the geodesic from $u$ to $v$ passing through the origin tends to $0$ as soon as both $\|u\|$ and $\|v\|$ tend to infinity.
In \cite{Ahlberg-Hanson-Hoffman}, they establish the following convergence result where $\cT_0$ is the tree of semi-infinite geodesics originating from $0$:
\[
\limsup_{\|v\| \to \infty} \P[ v \in \cT_0] = 0.
\]
The works in  \cite{Ahlberg-Hanson-Hoffman} and \cite{Ahlberg-Hoffman}  partly rely on ideas developed by Hoffman in \cite{Hoffman2} and Damron and Hanson in \cite{Damron-Hanson-14}.
See also \cite{Auffinger-Damron-Hanson-demi-plan}, \cite{Brito-Damron-Hanson} and \cite{Damron-Hanson-17} for this line of research.

Another direction of research is developed by Alexander in \cite{Alexander}.
He works under strong assumptions on fluctuations (governed by an exponent $\chi$) and curvature of the limiting shape.
He proves non existence of bigeodesics and derives  sharp quantitative results on coalescence of semi-infinite geodesics with the same direction.
In particular, when $d=2$, the tail behaviour for the distance to the coalescing point of two semi-infinite geodesics with the same asymptotic direction behaves
as $r^{-\xi}$ where $\xi=(1+\chi)/2$.
See \cite[Theorem 1.5]{Alexander}.

We refer to \cite{50} for a review of the literature on first-passage percolation up to 2015.

We turn to exponential two-dimensional last-passage percolation. This is an exactly solvable model.
We  single out a few results related to disjointness or coalescence of geodesics.
In \cite{Basu-Hoffman-Sly}, Basu, Hoffman and Sly prove that there exists no bi-infinite geodesic (apart from the trivial vertical and horizontal ones).
They also provide an upper bound on the existence of many disjoint geodesics of length of order $n$ starting and ending in intervals of length of order $n^{2/3}$.
This result is used by Basu, Sarkar and Sly in \cite{Basu-Sarkar-Sly} and by Zhang in \cite{Zhang}.
In \cite{Basu-Sarkar-Sly}, they provide  sharp estimates for the tail behaviour of the distance to coalescence of semi-infinite geodesics with the same direction (in $r^{-2/3}$).
They also give a power law upper bound on the tail behaviour of the same quantity for finite geodesics.
A  sharp estimate (in $r^{-2/3}$) of the same quantity is given in \cite{Zhang}.

\section{Proof of Item 1 of Theorems \ref{t:main-neighbor} and \ref{t:main-same}}

\subsection{Some preliminary results}

Let $H_n$ denote the hyperplane 
\[
H_n = \{(x_1,\dots,x_d) \in \R^d : x_1 = n/2\}=\left\{\frac{n}{2}\right\}\times\R^{d-1}.
\]
For $x\in \Z^d$, we consider geodesics from $x$ to $H_n$. 
In the case $n$ even, a geodesic from $x$ to $H_n$ is a path $\gamma$ from $x$ to some $y\in H_n\cap\Z^d$ that achieves the infimum
of $T(x,z)$ over all $z\in H_n\cap\Z^d$. In the case $n$ odd, geodesics from $x$ to $H_n$ has to be understand in the same way but after the following (invisible) transformation of the model: 
we add to $\Z^d$ the set of vertices $\{\frac{n}{2}\}\times \Z^{d-1}$, every edge $e$ between vertices of the form $(\frac{n-1}{2},x_2,\ldots,x_d)$ and $(\frac{n+1}{2},x_2,\ldots,x_d)$ is replaced by the two edges $e'=\{(\frac{n-1}{2},x_2,\ldots,x_d),(\frac{n}{2},x_2,\ldots,x_d)\}$ and $e''=\{(\frac{n}{2},x_2,\ldots,x_d),(\frac{n+1}{2},x_2,\ldots,x_d)\}$, and we set $\tau(e')=\tau(e'')=\tau(e)/2$. 
The choice of separating the time into two equal parts allows to preserve the symmetry of the environment on both sides of the hyperplane $H_n$. This is a central property in what follows.
Note that this transformation does not change any geodesic time or geodesic path between vertices of $\Z^d$.

Under \eqref{e:running_assumption}, the almost sure existence of point-to-hyperplane geodesics can be established in the same way as Proposition~4.4 in \cite{50}.  
For all $x\in \Z^d$, $\Gamma(x,H_n)$ denotes the set of all geodesics from $x$ to $H_n$, and to shorten it below, for all $i \in \Z$ we set $\Gamma_i=\Gamma(ie_2,H_n)$.

\begin{lemma} \label{l:controle-cardinal} Under \eqref{e:running_assumption}, there exists $C_1,C_2>0$ such that, for all $n \ge 1$,
\[
\P\left[\diam\left(\bigcup_{\gamma_0 \in \Gamma_0} \gamma_0\right) \le C_2n\right] \ge 1 - \exp(-C_1 n).
\]
\end{lemma}

This lemma is essentially Theorem 6.2 of \cite{Auffinger-Damron-Hanson-Adv-Math-2015}.
The only difference is that Theorem 6.2 deals with point to point geodesics.
The proof in the point to hyperplane case is almost identical.
We provide it in Section \ref{s:proof:l:controle-cardinal}.
Note that the result is straightforward if the support $\nu$ belongs to $[a,b]$ for some $b>a>0$ 
(In this case, each geodesic from $0$ to $H_n$ has travel time at mot $nb$ and therefore length at most $nb/a$.)

The following proposition is the key step and its proof relies on a very simple observation.

\begin{prop} \label{p:disjoint-half-space} Under \eqref{e:running_assumption}, there exists $C_3>0$ such that, for all $n \ge 1$ large enough,
\[
\P[\text{there exist $\gamma_0 \in \Gamma_0$ and $\gamma_1 \in \Gamma_1$ such that $\gamma_0 \cap \gamma_1 = \emptyset$}] \ge \frac{C_3}{n}.
\]
\end{prop}

\begin{proof} Fix $C_1, C_2$ as in the statement of Lemma \ref{l:controle-cardinal}.
Increasing $C_2$ if necessary, we can assume that $C_2$ is an integer.
For all $i,j \in \Z$, set
\[
A_n(i,j) = \{\text{there exist $\gamma_i \in \Gamma_i$ and $\gamma_j \in \Gamma_j$ such that $\gamma_i \cap \gamma_j = \emptyset$}\}.
\]
Note that we aim at giving a lower bound on $A_n(0,1)$.
Set $C = 2C_2  +1$ and consider the event
\[
M_n = \bigcap_{i=0}^{Cn-1} A_n(i,i+1)^c.
\]
Let us prove the inclusion 
\begin{equation}\label{e:key-observation}
M_n \subset \{\text{there exist $\gamma_0 \in \Gamma_0$ and $\gamma_{Cn} \in \Gamma_{Cn}$ such that $\gamma_0 \cap \gamma_{Cn} \neq \emptyset$}\}.
\end{equation}
Assume that $M_n$ occurs.
Fix $\gamma_0 \in \Gamma_0$ and denote by $s$ the endpoint of $\gamma_0$ on $H_n$.
Let us prove  by induction on $i$ the following property:
\begin{equation}\label{e:by_induction}
\text{for all $i \in \{0,\dots,Cn\}$, there exists $\gamma_i \in \Gamma_i$ such that $s$ is the endpoint $\gamma_i$}.
\end{equation}
Note that this will imply \eqref{e:key-observation}.
The property is true for $i=0$ by definition of $s$.
Let $i \in \{1,\dots,Cn-1\}$ be such that the property holds.
Let $\gamma_i \in \Gamma_i$ with endpoint $s$.
Let $\gamma_{i+1} \in \Gamma_i$.
As $A_n(i,i+1)$ occurs, there exists a vertex $c \in \Z^d$ which belongs to $\gamma_i$ and $\gamma_{i+1}$.
Denote by $\gamma'_{i+1}$ the concatenation of the path $\gamma_{i+1}$ from $(i+1)e_2$ to $c$ and of the path $\gamma_i$ from $c$ to $H_n$.
Then $\gamma'_{i+1} \in \Gamma_{i+1}$ and its endpoint is $s$.
Thus the required property holds for $i+1$ and \eqref{e:by_induction} follows by induction.
As \eqref{e:key-observation} is a consequence of \eqref{e:by_induction}, the proof of \eqref{e:key-observation}  is complete.

As $C=2C_2+1$, if
\[
\{\text{there exist $\gamma_0 \in \Gamma_0$ and $\gamma_{Cn} \in \Gamma_{Cn}$ such that $\gamma_0 \cap \gamma_{Cn} \neq \emptyset$}\} 
\]
holds, then there exists a path in $\Gamma_0$ or in $\Gamma_{Cn}$ whose diameter is strictly larger than $C_2n$.
Therefore \eqref{e:key-observation} and Lemma \ref{l:controle-cardinal} yield
\[
\P[M_n] \le 2 \exp(-C_1 n).
\]
In particular, for all $n$ large enough, $\P[M_n^c] \ge 1/2$.
For all such $n$, we thus have, by union bound and stationarity,
\[
\P[A_n(0,1)] \ge \frac{1}{2Cn}.
\]
This ends the proof.
\end{proof}

Item 1 of Theorem \ref{t:main-neighbor} is a consequence of Lemma \ref{l:controle-cardinal}, Proposition \ref{p:disjoint-half-space} and of a symmetrization  trick.
The proof of Item 1 of Theorem \ref{t:main-same} follows the same place but we need a strengthening of Proposition \ref{p:disjoint-half-space} which we state and prove now.

\begin{prop}\label{p:disjoint-half-space-same} Assume  \eqref{e:running_assumption} and \eqref{e:support-mu-0S}. There exists $C'_3>0$ such that, for all integer $n$ large enough, 
\[
\P[\text{there exist } \gamma, \gamma' \in \Gamma_0 \text{ such that } \gamma \cap \gamma' = \{0\}] \ge \frac{C'_3}n.
\]
\end{prop}

\begin{proof} 
Let $C_3>0$ be the constant given by Proposition \ref{p:disjoint-half-space}. 
We keep the notation $A_n(0,1)$ introduced in the proof of Proposition \ref{p:disjoint-half-space}.
By Proposition \ref{p:disjoint-half-space}, for all $n$ large enough, we thus have
\[
\P[A_n(0,1)] \ge \frac{C_3}n.
\]
Choose in an arbitrary deterministic way two disjoint geodesics $\gamma_0 \in \Gamma_0$ and $\gamma_1 \in \Gamma_1$.
Write $D(\gamma_0,\gamma_1) = \tau(\gamma_0)-\tau(\gamma_1)$.
Note that by \eqref{e:support-mu-0S} when $n$ is even $D(\gamma_0,\gamma_1)$ is necessarily an integer, but when $n$ is odd $D(\gamma_0,\gamma_1)$ can be a half-integer. 
This is due to our definition of point to $H_n$ geodesic times when $n$ is odd.
\begin{claim}\label{c:last} Reducing $C_3$ if necessary, we have, for all $n$ large enough,
\[
\P[A_n(0,1) \text{ and } D(\gamma_0,\gamma_1) \in \Z] \ge \frac{C_3}n.
\]
\end{claim}
The claim is straightforward when $n$ is even as $D(\gamma_0,\gamma_1)$ is always an integer.
In the odd case, the proof relies on a modification argument.
We will provide this argument below and, for now, assume the claim to be true.

By the triangle inequality we have
\[
| D(\gamma_0,\gamma_1) |\le \tau(0,e_2).
\]
Thus, by Claim \ref{c:last} and \eqref{e:support-mu-0S}, there exists $C_4>0$ and $\Delta\in \{-S,\dots,S\}$ such that
\begin{equation}\label{e:csq:p:disjoint-half-space}
\P[A_n(0,1) \text{ and } D(\gamma_0,\gamma_1)=\Delta] \ge \frac{C_4}{n}.
\end{equation}
In the sequel, we assume the event $A_n(0,1)\cap\{D(\gamma_0,\gamma_1)=\Delta\}$ occurs.
By symmetry,
\[
\P[A_n(0,1) \text{ and } D(\gamma_0,\gamma_1)=\Delta]=
\P[A_n(0,1) \text{ and } D(\gamma_0,\gamma_1)=-\Delta].
\]
Hence, we can assume $\Delta\ge 0$.
We define a new environment $(\overline\tau(e))_e$ by setting
\[
\overline\tau(e) = \begin{cases} \tau(e) & \text{ if } e \neq \{0,e_2\}, \\ \Delta & \text{ if } e  = \{0,e_2\}. \end{cases}
\]
We use the notations $\overline\tau, \overline T, \overline \Gamma$ and so on to denote quantities associated with these new edge times.
Distinguishing whether the path uses the edge $\{0,e_2\}$ or not, we get
\[
\overline T(0,H_n) \ge \min\big(T(0,H_n) , \Delta  + T(e_2,H_n) \big) = \min\big(\tau(\gamma_0), \Delta+\tau(\gamma_1)\big) = \tau(\gamma_0).
\]
Denote by $\gamma'_0$ the concatenation of the edge $(0,e_2)$ and of the path $\gamma_1$. 
We also have
\[
\overline \tau(\gamma_0) = \tau(\gamma_0) \text{ and } \overline\tau(\gamma'_0)=\Delta + \tau(\gamma_1) =\tau(\gamma_0).
\]
Therefore $\gamma_0$ and $\gamma'_0$ are two geodesics from $0$ to $H_n$ in the environment $(\overline\tau(e))_e$.
Moreover, by construction, $\gamma_0 \cap \gamma'_0 = \{0\}$.
This proves
\begin{align}\label{e:inclusion-modif}
A_n(0,1)& \cap \{D(\gamma_0,\gamma_1)=\Delta\}\nonumber\\
&\subset \big\{\text{there exist } \gamma_0, \gamma'_0 \in \Gamma_0 \text{ such that } \gamma_0 \cap \gamma'_0 = \{0\} \text{ in the environment } (\overline\tau(e))_e\big\}.
\end{align}
Finally, note that the edge times $(\overline\tau(e))_e$ can be obtained as follows: first resample independently $\tau(0,e_2)$ and then assume $\tau(0,e_2)=\Delta$.
Therefore, using \eqref{e:csq:p:disjoint-half-space} and \eqref{e:inclusion-modif} in the last step, we get
\begin{align*}
\P[ & \text{there exists } \gamma_0, \gamma'_0 \in \Gamma_0 \text{ such that } \gamma_0 \cap \gamma'_0 = \{0\}] \\
& \ge \P[\tau = \Delta]\P[\text{there exists } \gamma_0, \gamma'_0 \in \Gamma_0 \text{ such that } \gamma_0 \cap \gamma'_0 = \{0\} \text{ in the environment } (\overline\tau(e))_e] \\
& \ge \P[\tau = \Delta] \frac{C_4}{n}.
\end{align*}
Thanks to \eqref{e:support-mu-0S} this proves the result.

It remains to prove the claim.

\begin{proof}[Proof of Claim \ref{c:last}]
Note that under $A_n(0,1)\cap\{D(\gamma_0,\gamma_1)\not\in\Z\}$ (and thus $n$ odd), one of the geodesic times $\tau(\gamma_0)$ or $\tau(\gamma_1)$ is a half-integer and the other one is an integer.
By symmetry, 
\begin{equation}\label{e:Delta_not_integer}
 \P[A_n(0,1),\; D(\gamma_0,\gamma_1)\not\in\Z,\text{ and }\tau(\gamma_0)\not\in\N]=\frac12\P[A_n(0,1) \text{ and } D(\gamma_0,\gamma_1)\not\in\Z],
\end{equation}
where $\N$ denotes the set of non-negative integers.
We denote by $E_0$ the (random) edge crossing $H_n$ through the end-point of $\gamma_0$, that is $E_0$ is the edge $\{a,b\}$, $a,b\in\Z^d$ such that the end-point of $\gamma_0$ is $\frac12(a+b)\in H_n$. Note that, since $\tau(E_0)/2$ is the only possible half-integer time that is taking into account in $\tau(\gamma_0)$, we have that $\tau(\gamma_0)\not\in\N$ if and only if $\tau(E_0)$ is odd.
Assume the event 
\[
A_n(0,1)\cap\{D(\gamma_0,\gamma_1)\not\in\Z\}\cap\{E_0=e_0\}\cap\{\tau(e_0)=i_0\}
\]
occurs, where $e_0$ is some edge intersecting $H_n$ and $i_0$ is an odd integer in $\{1,\ldots,S\}$. 
We define a new environment $(\tau'(e))_{e}$ by setting
\[
\tau'(e) = \begin{cases} \tau(e) & \text{ if } e \neq e_0, \\ i_0-1 & \text{ if } e  = e_0, \end{cases}
\]
and we use the notation $\Gamma'(\cdot\,,\cdot)$ for sets of geodesics in this new environment.
We have 
\[
\tau'(\gamma_0)=\tau(\gamma_0)-\frac{\tau(e_0)}{2}+\frac{\tau'(e_0)}{2}=\tau(\gamma_0)-\frac12 \in \N. 
\]
Further, since the time of the edge $e_0$ is the only time that was decreased in this new environment, $\gamma_0\in\Gamma'(0,H_n)$ and any geodesic in $\Gamma'(0,H_n)$ has the same end-point as $\gamma_0$. 
Moreover, $\gamma_1$ also remains a geodesic from $e_2$ to $H_n$ in the environment $(\tau'(e))$. Indeed, if a path $\pi$ from $e_2$ to $H_n$ does not use the edge $e_0$, then $\tau'(\pi)=\tau(\pi)\ge \tau(\gamma_1)=\tau'(\gamma_1)$. If it uses the edge $e_0$, then $\tau(\pi)\in \frac12+\N$ and thus $\tau(\pi)\ge \tau(\gamma_1)+\frac12$ since $\tau(\gamma_1)\in\N$.
Hence $\tau'(\pi)=\tau(\pi)-\frac12\ge \tau(\gamma_1)=\tau'(\gamma_1)$.

We have shown that the event $A_n(0,1)$ still holds in the environment $(\tau'(e))$. If we denote by $\gamma_0'\in\Gamma'(0,H_n)$ and $\gamma_1'\in\Gamma'(e_2,H_n)$ the two disjoint geodesics chosen with the same deterministic rule as before and by $E'_0$ the edge containing the end-point of $\gamma_0'$, we have $E'_0=e_0$ and 
\[
D'(\gamma_0',\gamma_1'):=\tau'(\gamma_0')-\tau'(\gamma_1')=\tau(\gamma_0)-\frac12-\tau(\gamma_1)=D(\gamma_0,\gamma_1)-\frac12\in\Z.
\]
By reasoning as above, this proves that
\begin{align*}
 \P[ A_n(0,1),\; & D(\gamma_0,\gamma_1) \in\Z,\; E_0=e_0, \text{ and }\tau(e_0)=i_0-1] \\
&\ge  \P[\tau=i_0-1]\P[A_n(0,1),\;D(\gamma_0,\gamma_1)\not\in\Z,\;E_0=e_0,\text{ and }\tau(e_0)=i_0].
\end{align*}
Summing over all edges $e_0$ that intersect $H_n$ and all odd integers $i_0\in\{1,\ldots,S\}$, and using \eqref{e:support-mu-0S} and \eqref{e:Delta_not_integer}, we obtain a positive constant $C_5$ such that 
\begin{align*}
 \P[ A_n(0,1)\text{ and }D(\gamma_0,\gamma_1)\in\Z]
 &\ge C_5\P[A_n(0,1)\text{ and }D(\gamma_0,\gamma_1)\not\in\Z \text{ and } \tau(\gamma_0)\not\in\N]\\
 &= \frac{C_5}{2}\P[A_n(0,1)\text{ and }D(\gamma_0,\gamma_1)\not\in\Z].
\end{align*}
Reorganizing the above inequality and then using Proposition~\ref{p:disjoint-half-space} we get, for $n$ large enough,
\[
  \P[ A_n(0,1)\text{ and }D(\gamma_0,\gamma_1)\in\Z]\ge\frac{C_5}{2+C_5}\P[A_n(0,1)]\ge \frac{C_5}{2+C_5}\frac{C_3}{n}.
\]
This proves the claim.
\end{proof}
This ends the proof of the proposition.
\end{proof}

Finally, let us state the following immediate consequence of Cauchy-Schwarz inequality.

\begin{lemma} \label{l:cauchy-schwarz} Let $X$ be a random variable with value in a finite set $\cX$ of cardinality $N$. Then
\[
\sum_{x \in \cX} \P[X=x]^2 \ge \frac 1 N.
\]
\end{lemma}
\begin{proof}
By Cauchy-Schwarz inequality, $1 = \sum_{x \in \cX} \P[X=x] \le N \sum_{x \in \cX} \P[X=x]^2$. 
\end{proof}

\subsection{Proof of Item 1 of Theorem \ref{t:main-neighbor}}

\begin{proof}[Proof of Item 1 of Theorem \ref{t:main-neighbor}]
We fix $C_1, C_2$ as in Lemma \ref{l:controle-cardinal} and $C_3$ as in Proposition \ref{p:disjoint-half-space}.
Increasing $C_2$ if necessary, we can assume that $C_2$ is an integer.
Consider the good event
\[
G_n = \{ \text{there exists $\gamma_0 \in \Gamma_0$ and $\gamma_1 \in \Gamma_1$ such that $\gamma_0 \cap \gamma_1 = \emptyset$ and $|\gamma_0| \le C_2n$
and $|\gamma_1| \le C_2n$}\}.
\]
By Lemma \ref{l:controle-cardinal} and Proposition \ref{p:disjoint-half-space}, for all $n$ large enough,
\begin{equation}\label{e:csq-preliminary}
\P[G_n] \ge \frac{C_3}{2n}.
\end{equation}
When $G_n$ occurs, we choose in an arbitrary deterministic way (when several choices are possible) $\gamma_0$ and $\gamma_1$ as in the definition of the event $G_n$.
We denote by $U_0$ the endpoint of $\gamma_0$ and by $U_1$ the endpoint of $\gamma_1$. See Figure~\ref{f: Figure 1.}.
When $G_n$ does not occur, we can define $\gamma_0$, $\gamma_1$, $U_0$ and $U_1$ in an arbitrary way: this is irrelevant.
For short write $U=(U_0,U_1)$.

Symmetrically, we consider the geodesics from $ne_1$ to $H_n$ and from $ne_1+e_2$ to $H_n$.
With these geodesics we define in the same way as before an event $\overline G_n$, geodesics $\overline \gamma_0$ and $\overline \gamma_1$ 
and random variables $\overline U_0, \overline U_1$ and $\overline U$.
In particular, geodesics $\overline \gamma_0$ and $\overline \gamma_1$ are chosen in a deterministic way which is symmetric to the one used for $\gamma_0$ and $\gamma_1$,
so that $U$ and $\overline U$ have the same distribution.

\begin{figure}
	\begin{center}
		\begin{tikzpicture}[scale=0.5]
			\draw[line width=2pt,red] (12,-4.5) -- (12,6); 
			\draw[line width=1pt] (0,1) .. controls +(1,-1) and +(-2,0) .. (3,2.5) .. controls +(2,0) and +(-1,0) .. (6.5,3.2) .. controls +(1,0) and +(-1,-1) .. (9.5,3.2) .. controls +(1,1) and +(-1,0.5) .. (12,4);
			\draw[line width=1pt] (0,0) .. controls +(1,-1) and +(-0.5,0) .. (3,-0.5) .. controls +(0.75,0) and +(-01,0) .. (4.5,-2) .. controls +(1.5,0) and +(-1,0.5).. (7.2,-1.7) .. controls +(0.5,-0.25) and +(-0.5,-0.5) .. (9,-2) .. controls +(0.25,0.25) and +(-0.5,-0.25) .. (10.5,-0.5) .. controls +(0.5,0.25) and +(-1,0) .. (12,-0.5);
			\draw[line width=1pt] (12,2) .. controls +(1,0.4) and +(-4,2) .. (18.5,3.5) .. controls +(2,-1) and +(-1,0) .. (24,1);
			\draw[line width=1pt] (12,-2.5) .. controls +(1,0) and +(-1,0) .. (15,-0.5) .. controls +(1,0) and +(-1,0) .. (17.5,-1.1) .. controls +(1,0) and +(-1,-1) .. (20,0.5) .. controls +(1,1) and +(-1,0) .. (24,0);
			\draw (0,0) node {$\bullet$};
			\draw (0,1) node {$\bullet$};
			\draw (12,4) node {$\bullet$};
			\draw (12,2) node {$\bullet$};
			\draw (12,-0.5) node {$\bullet$};
			\draw (12,-2.5) node {$\bullet$};
			\draw (24,1) node {$\bullet$};
			\draw (24,0) node {$\bullet$};
			\draw (0,0) node[left] {$0$};
			\draw (0,1) node[left] {$e_2$};
			\draw (12,4) node[below left] {$U_1$};
			\draw (12,2) node[below right] {$\overline{U}_1$};
			\draw (12,-0.5) node[below left] {$U_0$};
			\draw (12,-2.5) node[below right] {$\overline{U}_0$};
			\draw (24,1) node[right] {$ne_1+e_2$};
			\draw (24,0) node[right] {$ne_1$};
			\draw[red] (12,-4.5) node[right] {$H_n$}; 
			\draw (4.5,-2) node[below] {$\gamma_0$};
			\draw (4.5,2.7) node[above] {$\gamma_1$};
			\draw (18.5,-0.9) node[below] {$\overline{\gamma}_0$};
			\draw (18.5,3.5) node[above] {$\overline{\gamma}_1$};
		\end{tikzpicture}
		\caption{When the events $G_n$ (on the left) and $\overline{G}_n$ (on the right) occur, if further $U=\overline{U}$, then the paths $\gamma_0 \cup \overline{\gamma}_0$ and $\gamma_1 \cup \overline{\gamma}_1$ are disjoint geodesics.}\label{f: Figure 1.}
	\end{center}
\end{figure}
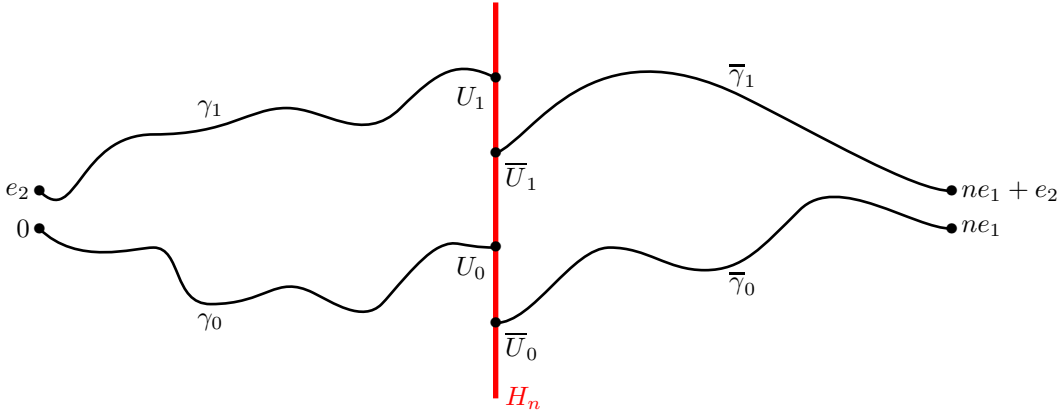

\begin{claim} \label{c:symmetrisation} The following inclusion holds:
\begin{align*}
G_n \cap \overline G_n  & \cap \{U=\overline U\} \nonumber \\
 & \subset \{\text{there exists $\gamma$ in $\Gamma(0,ne_1)$ and $\gamma'$ in $\Gamma(e_2,ne_1+e_2)$ such that $\gamma$ and $\gamma'$ are disjoint}\}.
\end{align*}
\end{claim}
\begin{proof}[Proof of Claim \ref{c:symmetrisation}]
Assume that the event on the left-hand side occurs.
The inequality $T(0,ne_1) \ge T(0,H_n) + T(ne_1,H_n)$ holds (this does not rely on the previous assumption).
But as $U_0=\overline U_0$, one can concatenate $\gamma_0$ and $\overline \gamma_0$ and this produces a path $\widetilde\gamma_0$ from $0$ to $ne_1$ with total passage time $T(0,H_n) + T(ne_1,H_n)$.
Therefore $T(0,ne_1) = T(0,H_n) + T(ne_1,H_n)$ and $\widetilde\gamma_0$ is a geodesic from $0$ to $ne_1$.
In a similar fashion, concatenating $\gamma_1$ and $\overline \gamma_1$ we get a path $\widetilde\gamma_1$ which is a geodesic from $e_2$ to $e_2+ne_1$.
As $\widetilde\gamma_0$ and $\widetilde\gamma_1$ are disjoint, the claim is proven.
\end{proof}

The proof of Item 1 of Theorem \ref{t:main-neighbor} is thus reduced to lower bounding $\P[G_n \cap \overline G_n  \cap \{U=\overline U\}]$. Consider the following subset of $\frac12\Z\times\Z^{d-1}$:
\[
\Lambda_n = \left\{\frac{n}{2}\right\} \times \{-C_2 n ,\dots , C_2  n\}^{d-1}.
\]

\begin{claim}\label{c:curry} For all $n \ge 1$,
\[
\P[G_n \cap \overline G_n  \cap \{U=\overline U\}] \ge  \P[G_n]^2 \frac{1}{|\Lambda_n|^2}.
\]
\end{claim}
\begin{proof}[Proof of Claim \ref{c:curry}]
The proof is slightly different here depending on whether $n$ is odd or even.
Let us first consider the case $n$ even. 
In this case there is independence between what concerns geodesics on either side of $H_n$.
(Note that when considering geodesics to $H_n$ the value of edges inside $H_n$ are irrelevant.)
When $G_n$ occurs, the random variables $U_0, U_1, \overline U_0$ and $\overline U_1$ take values in the set $\Lambda_n$ defined above.
Thus
\begin{align*}
\P[G_n \cap \overline G_n  \cap \{U=\overline U\}] 
& = \sum_{u_0, u_1 \in \Lambda_n}  \P[G_n \cap \{U=(u_0,u_1)\} \cap \overline G_n  \cap \{\overline U=(u_0,u_1)\}]. 
\end{align*}
Using the above mentioned independence alongside with symmetry, we get
\begin{align*}
\P[G_n \cap \overline G_n  \cap \{U=\overline U\}] 
& = \sum_{u_0, u_1 \in \Lambda_n}  \P[G_n \cap \{U=(u_0,u_1)\}]^2 \\
& = \P[G_n]^2 \sum_{u_0, u_1 \in \Lambda_n} \P[U=(u_0,u_1) | G_n]^2 \\
& \ge  \P[G_n]^2 \frac{1}{|\Lambda_n|^2}
\end{align*}
where we used Lemma \ref{l:cauchy-schwarz} in the last step.

Let us now consider the case $n$ odd. In this case we need to take care of the travel times of edges which intersect $H_n$.
Let us denote by $\cF_n$ the sigma-field generated by these random variables.
Denote by $\P_{\cF_n}$ the conditional probability given $\cF_n$.
Under $\P_{\cF_n}$, there is independence between what concerns geodesics on either side of $H_n$.
We thus have
\begin{align*}
\P[G_n \cap \overline G_n  \cap \{U=\overline U\}] 
& = \sum_{u_0, u_1 \in \Lambda_n}  \P[G_n \cap \{U=(u_0,u_1)\} \cap \overline G_n  \cap \{\overline U=(u_0,u_1)\}]  \\
& = \sum_{u_0, u_1 \in \Lambda_n}  \E\big[\P_{\cF_n}[G_n \cap \{U=(u_0,u_1)\} \cap \overline G_n  \cap \{\overline U=(u_0,u_1)\}]\big] \\
& = \sum_{u_0, u_1 \in \Lambda_n} \E\left[ [\P_{\cF_n}[G_n \cap \{U=(u_0,u_1)\}]^2\right] \\
& \ge \sum_{u_0, u_1 \in \Lambda_n}\P\left[G_n \cap \{U=(u_0,u_1)\}\right]^2
\end{align*}
where we used Jensen inequality in the last step, and we can conclude as in the case $n$ even above.
\end{proof}

We can now conclude the proof, using successively Claim \ref{c:symmetrisation}, Claim \ref{c:curry}, \eqref{e:csq-preliminary} and the definition of $\Lambda_n$:
\begin{align*}
 \P[\text{there}&\text{ exists $\gamma$ in $\Gamma(0,ne_1)$ and $\gamma'$ in $\Gamma(e_2,ne_1+e_2)$ such that $\gamma$ and $\gamma'$ are disjoint}] \\
& \ge \P[G_n \cap \overline G_n  \cap \{U=\overline U\}] \\
& \ge \P[G_n]^2 \frac{1}{|\Lambda_n|^2} \\
& \ge  \frac{C_3^2}{4(2C_2+1)^{2(d-1)} n^{2d}}.
\end{align*}
This ends the proof.
\end{proof}

\subsection{Proof of Item 1 of Theorem \ref{t:main-same}}

\begin{proof}[Proof of Item 1 of Theorem \ref{t:main-same}]
The proof is the same as in the previous section, using Proposition \ref{p:disjoint-half-space-same} instead of Proposition \ref{p:disjoint-half-space}. Here we consider the set 
\[
 G_n=\{\text{there exist } \gamma_0, \gamma_1 \in \Gamma_0 \text{ such that } \gamma_0 \cap \gamma_1 = \{0\},\; |\gamma_0|\le C_2n, \text{ and }|\gamma_1|\le C_2n\}.
\]
By Lemma \ref{l:controle-cardinal} and Proposition \ref{p:disjoint-half-space-same}, for $n$ large enough, $\P[G_n]\ge \frac{C_3'}{2n}$. Defining $\overline G_n$, $U$, and $\overline U$ as before, the same proof as for Claim \ref{c:symmetrisation} gives
\[
G_n \cap \overline G_n  \cap \{U=\overline U\}
 \subset \{\text{there exists $\gamma$ and $\gamma'$ in $\Gamma(0,ne_1)$ such that $\gamma\cap\gamma'=\{0,ne_1\}$}\},
\]
and we conclude as in Claim \ref{c:curry}. We skip the details.
\end{proof}

\section{Proof of Item 2 of Theorems \ref{t:main-neighbor} and \ref{t:main-same}}

\subsection{Coexistence in competition model of first-passage percolation}
\label{s:coex-intro}

\paragraph{Introduction.} Assume in this paragraph that the distribution of $\tau$ is absolutely continuous.
In this setting, for any $a,b \in \Z^d$ there exists a unique geodesic between $a$ and $b$ which we denote by $\gamma(a,b)$.
Define the random sets
\[
V = \{u \in \Z^d : T(0,u)<T(e_2,u)\} \text{ and } V' = \{u \in \Z^d : T(e_2,u)<T(0,u)\}.
\]
They form a partition of $\Z^d$.
Consider the coexistence event
\[
\text{coexistence} = \{\text{$V$ and $V'$ are infinite}\}.
\]
The model was introduced by Häggström and Pemantle in \cite{Haggstrom-Pemantle}. 
The name "coexistence" comes from the fact that $V$ and $V'$ can be interpreted as the sets of vertices ultimately infected by two competing infections.
The coexistence event is very closely related to the number of ends of the tree of geodesics\footnote{The tree of geodesics rooted at $0$ is the union over all $x \in \Z^d$ of the edges of $\gamma(0,x)$.}.
We refer to \cite{50} and \cite{Deijfen-Haggstrom-pleasures-pains} for review on this topic 
and to \cite{Ahlberg-review} for recent results on the link between coexistence and ends of the tree of geodesics.

Recall we assume that the law of $\tau$ is absolutely continuous.
Assume further that its support equals $[0,+\infty)$ and that it satisfies some mild integrability conditions.
Then,
\begin{equation}\label{e:coexistence_positive}
\P[\text{coexistence}]>0.
\end{equation}
The first proof appeared in \cite{Haggstrom-Pemantle} when $d=2$ and $\tau$ has an exponential distribution.
It was then extended independently to a general stationary setting by \cite{Hoffman1} and \cite{Garet-Marchand}.

\paragraph{General independent setting.} We get back to our usual setting: edge times are independent with distribution satisfying \eqref{e:running_assumption}.
In this setting, using ideas of \cite{Garet-Marchand, Hoffman1, Hoffman2} in the framework of regularized passage times developed by Cerf and Théret \cite{Cerf-Theret}
one can actually prove the following result
in which we get rid of any assumption except our running assumption \eqref{e:running_assumption}.
For all $x, y, z \in \Z^d$, we define
\[
B_z(x,y) = T(x,z) - T(y,z)
\]
and we set
\[
\overline\ell^+ = \limsup_{k \to \infty} B_{ke_2}(0,e_2) \text{ and } \underline\ell^- = \liminf_{k \to -\infty} B_{ke_2}(0,e_2).
\]
As $|B_z(x,y)| \le T(x,y)$ for any $x,y,z$, the previous quantities belong to $[-\tau(0,e_2), \tau(0,e_2)]$.

\begin{theorem} \label{t:coex-general} Under \eqref{e:running_assumption}, the probability $\P\big[\underline\ell^-<\overline\ell^+\big]$ is positive.
\end{theorem}

We provide the proof in Section \ref{s:proof:coex-general}.
Our proof of Item 2 of Theorems \ref{t:main-neighbor} and \ref{t:main-same} relies on Theorem \ref{t:coex-general}.

Note that $\{\underline\ell^-<\overline\ell^+\}$ is not the coexistence event. There is however a close link as shown by the following inclusion:
\[
\{\underline\ell^-< 0 < \overline\ell^+\} \subset \text{coexistence}.
\]
Indeed, when the event on the left hand-side occurs, there exist an infinite number of $k \ge 1$ such that $-ke_2 \in V$ and an infinite number of $k \ge 1$ such that $ke_2 \in V'$.
With a modification argument on can, under appropriate assumptions on the distribution of $\tau$, prove $\P[\underline\ell^-< 0 < \overline\ell^+]>0$ using $\P\big[\underline\ell^-<\overline\ell^+\big]>0$.
See the proof of Theorem 2 in \cite{Hoffman2} where such a modification argument is performed.

\subsection{Proof of Item 2 of Theorem \ref{t:main-neighbor}}

\begin{proof}[Proof of Item 2 of Theorem \ref{t:main-neighbor}]
Set
\[
\ell = \frac{\overline\ell^++\underline\ell^-}2.
\]
Assume that the event $\{\underline\ell^-<\overline\ell^+\}$ occurs.
Then the sets
\[
\overline V = \{z \in \Z^d : B_z(0,e_2) \ge \ell\} \text{ and } \underline V = \{z \in \Z^d : B_z(0,e_2) < \ell\}
\]
are both infinite.
For any $z \in \underline V$ and any geodesic $\gamma \in \Gamma(0,z)$, $\gamma$ is included in $\underline V$.
Indeed, for any vertex $w$ in $\gamma$,
\begin{equation}\label{e:gamma_in_V}
T(0,w) + T(w,z) = T(0,z) < T(e_2,z) + \ell \le T(e_2,w) + T(w,z) + \ell
\end{equation}
and thus $T(0,w) < T(e_2,w) + \ell$ that is $w \in \underline V$.
A symmetric statement holds for geodesics from $e_2$ to points of $\overline V$.
As both sets are infinite, this yields that for any $n \ge 1$ there exists $x,x' \in S_n$ such that $x$ belongs to $\underline V$ and $x'$ belongs to $\overline V$.
As moreover the union of  $\underline V$ and $\overline V$ is $\Z^d$, we deduce the existence of $y,y' \in S_n$ such that $y \sim y'$, $y$ belongs to $\underline V$ 
and $y'$ belongs to $\overline V$.
Consider any geodesics $\gamma \in \Gamma(0,y)$ and $\gamma' \in \Gamma(e_2,y')$.
The first one is included in $\underline V$.
The second one is included in $\overline V$.
As $\underline V$ and $\overline V$ are disjoint, we have proven
\begin{align*}
& \{\underline\ell^-<\overline\ell^+\} \\
& \subset \{\text{there exist } y\sim y'\in S_n \text{ such that there exist } \gamma \in \Gamma(0,y) \text{ and } \gamma' \in \Gamma(e_2,y') \text{ satisfying } \gamma \cap \gamma' = \emptyset\}.
\end{align*}
By union bound we get 
\begin{align*}
& \P\big[\underline\ell^-<\overline\ell^+\big] \\
& \le \sum_{y,y' \in S_n : y \sim y'}\P\Big[\text{there exist } \gamma \in \Gamma(0,y) \text{ and } \gamma' \in \Gamma(e_2,y') \text{ satisfying } \gamma \cap \gamma' = \emptyset\Big] \\
& \le C n^{d-1} \max_{y,y' \in S_n : y \sim y'}\P\Big[\text{there exist } \gamma \in \Gamma(0,y) \text{ and } \gamma' \in \Gamma(e_2,y') \text{ satisfying } \gamma \cap \gamma' = \emptyset\Big]
\end{align*}
for some constant $C=C(d)>0$.
Item 2 of Theorem \ref{t:main-neighbor} follows from Theorem \ref{t:coex-general}.
\end{proof}

\subsection{Proof of Item 2 of Theorem \ref{t:main-same}}

In addition to our running assumption \eqref{e:running_assumption}, we assume \eqref{e:support-mu-0S}.

\begin{claim} \label{c:coex-K}
There exists an integer $K\in\{0,1,\ldots,S\}$ such that the probability of the event
\begin{equation} \label{e:coex-K}
\Big\{\text{the sets } \{z \in \Z^d : B_z(0,e_2) \ge K\} \text{ and } \{z \in \Z^d : B_z(0,e_2)<K\} \text{ are infinite}\Big\}
\end{equation}
is positive.
\end{claim}

\begin{proof}
This is an immediate consequence of Theorem~\ref{t:coex-general} and of the symmetry of the model.
Note
\begin{equation}\label{e:sym-limsup-liminf}
\liminf_{\|z\|\to\infty} B_z(0,e_2) =  - \limsup_{\|z\|\to\infty} B_z(e_2,0).
\end{equation}
By Theorem~\ref{t:coex-general} we thus get
\[
\P\left[\limsup_{\|z\|\to\infty} B_z(e_2,0) + \limsup_{\|z\|\to\infty} B_z(0,e_2) > 0\right]>0.
\]
On the event described above, one of the two $\limsup$ is positive.
By symmetry, we can therefore choose a non-negative integer $K$ such that
\[
\P\left[\limsup_{\|z\|\to\infty} B_z(e_2,0) + \limsup_{\|z\|\to\infty} B_z(0,e_2) > 0 \text{ and } \limsup_{\|z\|\to\infty} B_z(0,e_2) = K\right]>0.
\]
Using \eqref{e:sym-limsup-liminf} back we get that this new event is 
\[
\left\{\limsup_{\|z\|\to\infty} B_z(0,e_2) = K \text{ and } \liminf_{\|z\|\to\infty} B_z(0,e_2) < K\right\}.
\]
By \eqref{e:support-mu-0S} the claim follows.
\end{proof}

We fix such an integer $K$ for the rest of the proof. We define for all path $\pi$, $\tau_K(\pi)=K+\tau(\pi)$ and for all $x,y\in\Z^d$, $T_K(x,y)=K+T(x,y)$.

\begin{claim} \label{c:coex-K-yyprime-Delta}
There exists $C>0$ such that, for all $n \ge 1$, there exist $\Delta \in \{-S,\dots,S\}$ and $y,y' \in S_n$ satisfying $y \sim y'$ and
\[
\P\big[T(0,y) < T_K(e_2,y), \; T(0,y') \ge T_K(e_2,y'), \; T(0,y)=T_K(e_2,y')+\Delta 
\big] \ge \frac{C}{n^{d-1}}.
\]
\end{claim}
\begin{proof} The claim follows from the following inequality by choosing some optimal $y,y'$ and $\Delta$:
\begin{align} \label{e:intermediaire:c:coex-K-yyprime-Delta}
\P[\text{event }\eqref{e:coex-K}]  
 \le \P\Big[& \text{there exists } y \sim y' \in S_n \text{ and } \Delta \in \{-S,\dots,S\} \text{ such that } \nonumber \\
& T(0,y) < T_K(e_2,y) \text{ and } T(0,y') \ge T_K(e_2,y') \text{ and } T(0,y)=T_K(e_2,y')+\Delta\Big].
\end{align}
We now aim at proving \eqref{e:intermediaire:c:coex-K-yyprime-Delta}.
Assume that the event \eqref{e:coex-K} occurs. 
Choose $z$ such that $B_z(0,e_2) < K$ and $\|z\|_\infty \ge n$.
Consider a geodesic $\gamma$ from $0$ to $z$. By the same considerations as in \eqref{e:gamma_in_V}, all points $w$ of $\gamma$ satisfy $B_w(0,e_2)<K$.
In particular, there exists $u \in S_n$ such that $B_u(0,e_2)<K$.
Similarly, considering a geodesic from $0$ to a point of $\{z' : B_{z'}(0,e_2) \ge K\}$ such that $\|z'\|_\infty \ge n$, 
we get the existence of $u' \in S_n$ such that $B_{u'}(0,e_2) \ge K$.
Following a path inside $S_n$ from $u$ to $u'$, we finally get the existence of $y \sim y'$ on $S_n$ such that
$B_y(0,e_2)<K$ and $B_{y'}(0,e_2) \ge K$.
Set $\Delta = T(0,y) - T_K(e_2,y') \in \Z$.
Let us check $|\Delta| \le \tau(y,y')$.
This is a consequence of the following inequalities:
\[
T_K(e_2,y') \le T(0,y') \le T(0,y)+\tau(y,y') \text{ and } T(0,y) \le T_K(e_2,y) \le T_K(e_2,y')+\tau(y,y').
\]
Hence, by \eqref{e:support-mu-0S}, $|\Delta| \le S$.  This ends the proof of \eqref{e:intermediaire:c:coex-K-yyprime-Delta} and thus the proof of the claim.
\end{proof}

\begin{claim}\label{c:coex-y} There exists $C'>0$ such that, for all $n \ge 1$, there exists $u \in S_n$ such that
\[
\P\Big[T(0,u)=T_K(e_2,u) \text{ and there exists } \gamma^1 \in \Gamma(0,u) \text{ and } \gamma^2 \in \Gamma(e_2,u) \text{ such that } \gamma^1 \cap \gamma^2 = \{u\}\Big] \ge \frac{C'}{n^{d-1}}.
\]
\end{claim}
\begin{proof} Let $n \ge 1$. Let $\Delta,y,y'$ be given by Claim \ref{c:coex-K-yyprime-Delta}.
We aim at proving
\begin{equation}\label{e:lien-yyprime-u}
\P[\text{event of Claim } \ref{c:coex-K-yyprime-Delta}] \P[\tau(y,y')=\Delta] \le \P[\text{event of Claim } \ref{c:coex-y}]
\end{equation}
for some appropriate $u$ which will be $y$ or $y'$ depending on whether $\Delta$ is non-negative or not.
Once \eqref{e:lien-yyprime-u} is proven, the claim follows by Claim \ref{c:coex-K-yyprime-Delta} and by \eqref{e:support-mu-0S}.

Assume that the event of Claim \ref{c:coex-K-yyprime-Delta} occurs. 
We will use repeatedly in the course of the proof, and without explicitly stating it, the properties guaranteed by this event.
Fix $\gamma \in \Gamma(0,y)$ and $\gamma' \in \Gamma(e_2,y')$ such that $\gamma \cap \gamma' = \emptyset$.
There are two cases depending on whether $\Delta$ is non-negative or not.
The proof is essentially the same in each case.

Let us first assume $\Delta \ge 0$. 
We set $\gamma^1=\gamma$ and we define $\gamma^2$ as the concatenation of the paths $\gamma'$ and $(y',y)$.
Define new edge times $(\overline\tau(e))_e$ by setting 
\[
\overline\tau(e) = \begin{cases} \tau(e) & \text{ if } e \neq \{y,y'\}, \\ \Delta & \text{ if } e  = \{y,y'\}. \end{cases}
\]
We use the notations $\overline\tau, \overline\tau_K, \overline T, \overline T_K$ and $\overline \Gamma$ to denote quantities associated with these new edge times.
First note
\[
\overline\tau(\gamma^1)=\tau(\gamma^1)=T(0,y) \text{ and } \overline\tau_K(\gamma^2)=\tau_K(\gamma')+\Delta = T_K(e_2,y')+\Delta=T(0,y).
\]
Distinguishing whether paths uses the edge $\{y,y'\}$ or not, we get
\[
\overline T(0,y) \ge \min\big(T(0,y) , T(0,y') + \Delta\big) \ge \min\big(T(0,y) , T_K(e_2,y') + \Delta\big)=T(0,y).
\]
As $\overline\tau(\gamma^1) = T(0,y)$ and as $\gamma^1$ is a path from $0$ to $y$, we get
\[
\overline\tau(\gamma^1) = \overline T(0,y) = T(0,y) \text{ and } \gamma^1 \in \overline\Gamma(0,y).
\]
Similarly,
\[
\overline T_K(e_2,y) \ge \min\big(T_K(e_2,y) , T_K(e_2,y') + \Delta\big) = T(0,y).
\]
As $\overline\tau_K(\gamma^2) = T(0,y)$ and as $\gamma^2$ is a path from $e_2$ to $y$, we get
\[
\overline\tau_K(\gamma^2) =\overline T_K(e_2,y) = T(0,y) \text{ and } \gamma^2 \in \Gamma(e_2,y).
\]
Setting $u=y$ we thus have, when $\Delta \ge 0$,
\begin{align}
& \{\text{event of Claim }\ref{c:coex-K-yyprime-Delta}\} \nonumber \\
& \subset 
\Big\{\overline T(0,u)=\overline T_K(e_2,u) \text{ and there exists } \gamma^1 \in \overline\Gamma(0,u) \text{ and } \gamma^2 \in \overline\Gamma(e_2,u) 
\text{ such that } \gamma \cap \gamma' = \{u\} \Big\}. \label{e:inclusion-yyprime-u}
\end{align}
Finally, note that the edge times $(\overline\tau(e))_e$ can be obtained as follows: first resample independently $\tau(y,y')$ and then assume $\tau(y,y')=\Delta$.
We thus get \eqref{e:lien-yyprime-u} with $u=y$ which concludes the proof in the case $\Delta \ge 0$.

Let us sketch the proof in the case $\Delta < 0$.
We define $\gamma^1$ as the concatenation of $\gamma$ and $(y,y')$ and we set $\gamma^2=\gamma'$.
Define new edge times $(\overline\tau(e))_e$ by setting 
\[
\overline\tau(e) = \begin{cases} \tau(e) & \text{ if } e \neq \{y,y'\}, \\ -\Delta & \text{ if } e  = \{y,y'\}. \end{cases}
\]
We have, arguing as above,
\begin{align*}
\overline\tau(\gamma^1) & =T(0,y)-\Delta=T_K(e_2,y'), \\
\overline\tau_K(\gamma^2)& =T_K(e_2,y'), \\
\overline T(0,y') & \ge \min\big(T(0,y') , T(0,y) - \Delta\big)  = T_K(e_2,y'), \\
\overline T_K(e_2,y') &\ge \min\big(T_K(e_2,y'),T_K(e_2,y)-\Delta) \ge T_K(e_2,y'). 
\end{align*}
We thus get \eqref{e:inclusion-yyprime-u} with $u=y'$ and we conclude as before.
\end{proof}

\begin{claim}\label{c:final-coex-chirurgie} There exists $C''>0$ such that, for all $n \ge 1$, 
\[
\P[\text{there exists } \gamma^1 \in \Gamma(0,u) \text{ and } \gamma^2 \in \Gamma(0,u) \text{ such that } \gamma^1 \cap \gamma^2 = \{0,u\}\Big] \ge \frac{C''}{n^{d-1}}.
\]
\end{claim}
\begin{proof} Let $n \ge 1$. Let $u$ be given by Claim \ref{c:coex-y}.
We aim at proving
\begin{equation}\label{e:inegalite-u-final}
\P[\text{event of Claim }\ref{c:coex-y}] \P[\tau=K] \le \P[\text{event of Claim } \ref{c:final-coex-chirurgie}].
\end{equation}
By Claim \ref{c:coex-y} and \eqref{e:support-mu-0S} this will conclude the proof.

Assume that the event of Claim \ref{c:coex-y} occurs.
Let $\gamma^1$ and $\gamma^2$ be as in this event.
Now define new edge times $(\overline\tau(e))_e$ by setting 
\[
\overline\tau(e) = \begin{cases} \tau(e) & \text{ if } e \neq \{0,e_2\}, \\ K & \text{ if } e  = \{0,e_2\}. \end{cases}
\]
Define $\gamma^3$ as the concatenation of $(0,e_2)$ and of $\gamma^2$.
We get, distinguishing whether the path uses the edge $\{0,e_2\}$ or not for the first inequality, 
\begin{align*}
\overline T(0,u) & \ge \min\big(T(0,u),T_K(e_2,u)\big) = T(0,u), \\
\overline \tau(\gamma^1) & = \tau(\gamma^1) = T(0,u), \\
\overline \tau(\gamma^3) & = \tau_K(\gamma^2) = T_K(e_2,u) = T(0,u).
\end{align*}
Therefore $\gamma^1$ and $\gamma^3$ belongs to $\overline \Gamma(0,u)$.
Moreover $\gamma^1 \cap \gamma^3 = \{0,u\}$.
Arguing as in the proof of Claim \ref{c:coex-y}, we get \eqref{e:inegalite-u-final} which concludes the proof.
\end{proof}

\appendix

\section{Proof of Lemma \ref{l:controle-cardinal}}

\label{s:proof:l:controle-cardinal}

We follow very closely the proof of Theorem 6.2 in \cite{Auffinger-Damron-Hanson-Adv-Math-2015} which is the analogous of Lemma \ref{l:controle-cardinal} for point to point geodesics.

The proof requires some results about Bernoulli bond percolation on $\Z^d$.
Let $p \in [0,1]$.
Let $(\omega(e))_{e \in \cE}$ be a family of independent random variables with distribution Bernoulli of parameter $p$ indexed by the set $\cE$ of edges of $\Z^d$.
Consider the random sub-graph of $\Z^d$ whose vertex set is $\Z^d$ and whose edge set is $\{e \in \cE : \omega(e)=1\}$.
It is known that for $p$ strictly larger than the critical threshold $p_c \in (0,1)$ this random sub-graph contains a unique infinite component.
We call it the infinite cluster and denote it by $I$.
We write $d_I$ for the intrinsic graph distance on $I$.
For all $n \ge 1$ we denote by $B_n$ the ball $B_n=\{-n,\dots,n\}^d$.
We need the following two lemmas about percolation.

\begin{lemma}[Lemma 6.3 in\cite{Auffinger-Damron-Hanson-Adv-Math-2015}] \label{l:rencontre_I} There exists $p_0 \in (p_c,1)$ such that, for any $p \in [p_0,1]$, there exists $C>0$ such that, for all $n \ge 1$,
\[
\P[\text{any path from $0$ to $S_n$ intersects the infinite cluster}\,] \ge 1-\exp(-Cn).
\]
\end{lemma}

\begin{lemma}[Theorem 1.1 in \cite{Antal-Pisztora}] \label{l:AP} Let $p>p_c$. There exists a constant $\rho \in [1,\infty)$ such that
\[
\limsup_{\|y\|_\infty \to \infty} \frac{1}{\|y\|_\infty} \ln \P[d_I(0,y) \ge \rho \|y\|_\infty, \; 0 \in I, \; y \in I]<0.
\]
\end{lemma}

We also need the following result about first-passage percolation. It holds under our running assumption. 

\begin{lemma}[Proposition 5.8 in \cite{Kesten-St-Flour}] \label{l:5.8-Saint-Flour}
Assume \eqref{e:running_assumption}.
There exists $\alpha>0$ and $C>0$ such that, for all $n \ge 1$
\[
\P[\text{there exists a self-avoiding path $\gamma$ starting at $0$ with $|\gamma| \ge  n$ and $T(\gamma) \le \alpha n$}] \le e^{-Cn}.
\]
\end{lemma}

We are now ready to give the proof of Lemma \ref{l:controle-cardinal}.

\begin{proof}[Proof of Lemma \ref{l:controle-cardinal}] Let $p_0$ be given by Lemma \ref{l:rencontre_I}.
Let $M>0$ be such that $\P[\tau \le M] \ge p_0$.
We consider the Bernoulli bond percolation on $\Z^d$ defined by $\omega(e)=\1_{\tau(e) \le M}$ for all edge $e\in \cE$.
By construction, the parameter $p=\P[\tau \le M]$ of this percolation process satisfies $p \ge p_0>p_c$.
Let $\rho \ge 1$ be given by Lemma \ref{l:AP}.

Let $n \ge 1$. 
We implicitly assume that $n$ is large enough at various places.
Consider the boxes
\[
D = B_{\lfloor n/10 \rfloor} \text{ and } \overline D = ne_1 + B_{\lfloor n/10 \rfloor}.
\]
We now define several good events:
\begin{align*}
F_1 & = \{\text{any path from $0$ to $\partial D$ contains at least one vertex in $I$}\}, \\
\overline F_1 & = \{\text{any path from $ne_1$ to $\partial \overline D$ contains at least one vertex in $I$}\}, \\
F_2 & = \{\text{for all $x\in D \cap I$ and $\overline x \in \overline D \cap I$, there exists a path $\pi$ in $I$ from $x$ to $\overline x$ of length at most $2\rho n$}\}, \\
F_3 & = \{\text{for all $x \in D$ and all path $\pi$ starting from $x$, if $|\pi| \ge 3\rho M \alpha^{-1} n$, then $T(\pi) > 2\rho M n$}\}.
\end{align*}

\begin{figure}
	\begin{center}
		\begin{tikzpicture}[scale=0.45]
			\draw[fill=gray!20] (-2,-2) rectangle (2,2);
			\draw[fill=gray!20] (22,-2) rectangle (26,2);
			\draw[line width=2pt,red] (12,-4.5) -- (12,6); 
			\draw[line width=1pt] (0,0) .. controls +(0.5,0.5) and +(0,-0.5) .. (0.5,1.5) .. controls +(0,1.5) and +(-1.5,-0.5) .. (3,4) .. controls +(3,1) and +(-1.5,0) .. (7,4.2) .. controls +(1,0) and +(-2,1) .. (12,4.2);
			\draw[line width=1pt,blue] (0.5,1.5) .. controls +(1,-1) and +(-1,0) .. (3,0.5) .. controls +(2,0) and +(-2.5,0) .. (7,1.5) .. controls +(3,0) and +(-4,0) .. (15,-1) .. controls +(1,0) and +(-2,0) .. (20,0) .. controls +(1,0) and +(-1,1) .. (23,-1);
			\draw[line width=1pt] (24,0) .. controls +(0,-0.5) and +(0.5,0) .. (23,-1) .. controls +(-0.5,0) and +(0,0.5) .. (22.5,-2);
			\draw (0,0) node {$\bullet$};
			\draw (0.5,1.5) node {$\bullet$};
			\draw (12,4.2) node {$\bullet$};
			\draw (23,-1) node {$\bullet$};
			\draw (24,0) node {$\bullet$};
			\draw (0,0) node[left] {$0$};
			\draw (0.5,1.5) node[left] {$x$};
			\draw (23,-1) node[below right] {$\overline{x}$};
			\draw (24,0) node[below right] {$ne_1$};
			\draw (7,4.6) node {$\gamma_0$};
			\draw[blue] (10,0.5) node[below] {$\pi$};
			\draw (2.2,-2) node[below left] {$D$};
			\draw (26.2,-2) node[below left] {$\overline{D}$};
			\draw[red] (12,-4.5) node[right] {$H_n$};
        \end{tikzpicture}
		\caption{The path $\pi$ is contained in the infinite cluster $I$ and has controlled length.}\label{f: Figure 2.}
	\end{center}
\end{figure}
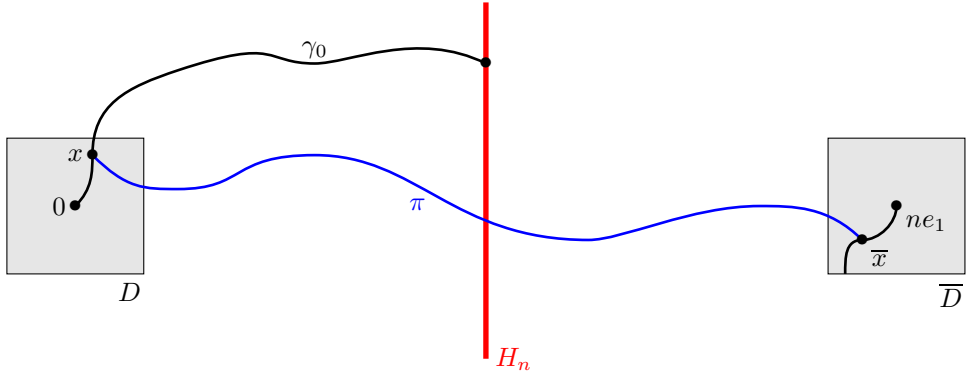

\begin{claim} \label{c:probaF} Set $F:=F_1 \cap \overline F_1 \cap F_2 \cap F_3$.
There exists $C>0$ such that, for all $n$ large enough, 
\[
\P[F] \ge 1-\exp(-Cn).
\]
\end{claim}
\begin{proof}[Proof of the claim]
We lower bound $\P[F_1]$ and $\P[\overline F_1]$ using Lemma \ref{l:rencontre_I}, $\P[F_2]$ thanks to Lemma \ref{l:AP} and $\P[F_3]$  with the help of Lemma \ref{l:5.8-Saint-Flour}.
\end{proof}

\begin{claim} On $F$, 
\[
\diam\left(\bigcup_{\gamma_0 \in \Gamma_0} \gamma_0\right) \le 2(3\rho M\alpha^{-1}+1)n.
\]
\end{claim}

\begin{proof}[Proof of the claim] The proof is illustrated by Figure~\ref{f: Figure 2.}.
Assume that $F$ occurs.
Let $\gamma_0 \in \Gamma_0$.
As $F_1$ occurs, we can fix a vertex $x \in D$ which belongs to $I$ and $\gamma_0$ such that the restriction of $\gamma_0$ between $0$ and $x$ remains inside $D$.
As $\overline F_1$ occurs, we can fix a vertex $\overline x \in \overline D$ which belongs to $I$ (use any deterministic path from $ne_1$ to $\partial \overline D$).
Using the occurrence of $F_2$ we get the existence of a path $\pi$ from $x$ to $\overline x$ which remains inside $I$ and such that $|\pi| \le 2\rho n$.
By definition of our percolation process, each edge $e$ of $\pi$ satisfies $\tau(e) \le M$.
Therefore $\tau(\pi) \le 2\rho M n$.
But $\pi$ starts at $x$, ends at $\overline x$ and these two points are on two different sides of $H_n$.
Therefore $\pi$ crosses $H_n$ and thus 
\[
T(x,H_n) \le \tau(\pi) \le 2\rho M n.
\]
Consider now the restriction of $\gamma_0$ from $x$ to $H_n$.
Let us call it $\gamma_x$.
This is a geodesic from $x$ to $H_n$. Therefore 
\[
\tau(\gamma_x) \le T(x,H_n) \le 2\rho M n.
\]
We have to take into account that for $n$ odd the last vertex of $\gamma_x$ does not belong to $\Z^d$.
We denote by $y$ the last vertex of $\gamma_x$ that belongs to $\Z^d$ which is the last vertex of $\gamma_x$ when $n$ is even or the second to last when $n$ is odd.
The restriction $\gamma_{x,y}$ of $\gamma_x$ from $x$ to $y$ only contains vertices in $\Z^d$ and
we have  
\[
\tau(\gamma_{x,y}) \le \tau(\gamma_x) \le 2\rho M n.
\] 
As $F_3$ occurs, we deduce $|\gamma_{x,y}| \le 3\rho M \alpha^{-1} n$.
Further, $y$ is at most at $\|\cdot\|_\infty$-distance $1/2$ from $H_n$ and thus
\[
\diam( \gamma_x )\le  3\rho M \alpha^{-1} n+\frac12.
\]
Finally, as the restriction of $\gamma_0$ between $0$ and $x$ remains inside $D$, we get 
\[
\gamma_0 \subset B_{(3\rho M \alpha^{-1}+1)n}.
\]
The claim follows. 
\end{proof}
Lemma \ref{l:controle-cardinal} follows from the two previous claims.
\end{proof}

\section{Proof of Theorem \ref{t:coex-general}}\label{s:proof:coex-general}

The proof we present is essentially contained in \cite{Hoffman2}.
In particular, Lemma 4.4 in \cite{Hoffman2} corresponds to the key Proposition \ref{p:coex} below and their proofs are very close.
We use the framework of regularized passage times developed by Cerf and Théret \cite{Cerf-Theret} which provides almost for free strong integrability properties.

\subsection{Regularized passage times}

We describe here what we need and refer to \cite{Cerf-Theret} for a more detailed description.
We take $M$ sufficiently large so that the edges with passage time $\tau(e) \le M$ percolate. We denote by $\cC^M$ the unique infinite component.
For every $x \in \Z^d$, let $x^M$ be a point chosen among the points in $\cC^M$ that are the closest to $x$ in the $\|\cdot\|_1$ norm.
If there are several such points, we choose one according to an arbitrary (but translation-invariant) deterministic rule.
For all $x, y \in \Z^d$, we set
\[
T^M(x,y) := T(x^M,y^M).
\]
One of the major advantages of these regularized times is that they are integrable.
This is an immediate consequence of Proposition 1 in \cite{Cerf-Theret}.
In this context, Cerf and Théret establish the following result.

\begin{theorem}[\cite{Cerf-Theret}] Let $M$ be such that $\P[\tau \le M]>p_c$. There exists a deterministic norm $\mu$ on $\R^d$ such that
\[
\lim_{\|x\|_1 \to \infty} \frac{T^M(0,x)}{\mu(x)} = 1 \quad \text{a.s. and in }L^1.
\]
This norm does not depend on the choice of $M$.
\end{theorem}
\begin{proof} We provide the proof as it is necessary to combine several results from \cite{Cerf-Theret}.
Note that our times $T^M$ are denoted by $\widetilde T$ in \cite{Cerf-Theret}.
Theorem 1 of \cite{Cerf-Theret} ensures the existence of a deterministic function $\mu$ (denoted $\widetilde\mu$ in \cite{Cerf-Theret}) such that for all $x\in\Z^d$, $\lim_{n\to\infty} \frac{T^M(0,nx)}{n}=\mu(x)$ a.s. and in  $L^1$.
Theorems 1 and 2 of \cite{Cerf-Theret} (see the comments following Theorem 2 of \cite{Cerf-Theret}) ensure that $\mu$ is a norm.
By Theorem 4 of \cite{Cerf-Theret}, $\mu$ does not depend on the choice of $M$.
The first part of Theorem 3 of \cite{Cerf-Theret} states (using the equivalence of norms)
\begin{equation}\label{e:CT}
\lim_{n \to \infty} \sup_{x \in \Z^d : \|x\|_1 \ge n} \left| \frac{T^M(0,x)}{\mu(x)}-1\right| = 0 \quad \text{a.s.}
\end{equation}
The random variables above are dominated by $aS+1$ for some constant $a>0$ with
\[
S = \sup_{x \in \Z^d : \|x\|_1 \ge 1} \frac{T^M(0,x)}{\|x\|_1}.
\]
To obtain $L^1$ convergence in \eqref{e:CT}, it suffices to show that $S$ is integrable.
Using Proposition 1 of \cite{Cerf-Theret} and adopting its notations, for any $u \ge 0$ we have
\begin{align*}
\P[S > C_3+u] 
& \le \sum_{x \in \Z^d : \|x\|_1 \ge 1} \P[T^M(0,x) > C_3\|x\|_1 + u \|x\|_1] \\
& \le \sum_{x \in \Z^d : \|x\|_1 \ge 1} C_1 e^{-C_2(C_3\|x\|_1 + u \|x\|_1)} \\
& \le \sum_{x \in \Z^d : \|x\|_1 \ge 1} C_1 e^{-C_2(C_3\|x\|_1 + u)}.
\end{align*}
This quantity is summable (in $u \in \N$).
Thus, $S$ is integrable and the convergence in \eqref{e:CT} also holds in $L^1$.
The theorem follows (we have shown a stronger result).
\end{proof}

\subsection{Control of time differences}

For all $x, y, z \in \Z^d$, we define
\[
B_z(x,y) = T(x,z) - T(y,z) \text{ and } B_z^M(x,y) = T^M(x,z) - T^M(y,z).
\]
We will repeatedly use without comments the triangle inequalities
\begin{equation}\label{e:B-triangle}
|B_z(x,y)| \le T(x,y) \text{ and } |B^M_z(x,y)| \le T^M(x,y).
\end{equation}
We will also use antisymmetry in $(x,y)$ and additivity:
\begin{equation}\label{e:B-basic}
B_z(x,y)=-B_z(y,x) \text{ and } B_z(w,x)+B_z(x,y) = B(w,y).
\end{equation}
Recall that $T^M(x,y)$ is integrable by Proposition 1 in \cite{Cerf-Theret}.
By \eqref{e:B-triangle}, this ensures that $B^M_z(x,y)$ is always integrable.

The main objective of this section is to show the following proposition and its corollary.
The proposition is essentially Lemma 4.4 of [Hof2].
The proof closely follows the one given by Hoffman in [Hof2].

\begin{prop} \label{p:coex} Let $\epsilon > 0$. Let $M$ be such that $\P[\tau \le M] > p_c$.
There exists $A > 0$ such that, for all $x \in \Z^d$ satisfying $\|x\|_1 \ge A$ and for all $N \ge 1$, we have
\[
\frac{1}{N} \#\Big\{ k \in \{1, \dots, N\} : \P[B^M_{kx}(0,x) \ge \mu(x)(1 - \epsilon)] \ge 1 - \epsilon \Big\} \ge 1 - \epsilon.
\]
In particular, for all $x \in \Z^d$ satisfying $\|x\|_1 \ge A$,
\begin{equation}\label{e:p:coex:infini}
\P\big[ \{k \ge 1 : B^M_{kx}(0,x) \ge \mu(x)(1 - \epsilon)\} \text{ is infinite} \,\big] \ge 1 - \epsilon.
\end{equation}
\end{prop}

We will use the proposition through the corollary we now state.
Note that the corollary concerns $B$, while the proposition concerns $B^M$.

\begin{corollary} \label{c:coex} Let $\epsilon > 0$. There exists $A > 0$ such that,
for all $x \in \Z^d$ satisfying $\|x\|_1 \ge A$ and for all $N \ge 1$, we have
\[
\frac{1}{N} \#\Big\{ k \in \{1, \dots, N\} : \P[B_{kx}(0,x) \ge \mu(x)(1 - \epsilon)] \ge 1 - \epsilon \Big\} \ge 1 - \epsilon.
\]
In particular, for all $x \in \Z^d$ satisfying $\|x\|_1 \ge A$,
\begin{equation}\label{e:c:coex:infini}
\P\big[ \{k \ge 1 : B_{kx}(0,x) \ge \mu(x)(1 - \epsilon)\} \text{ is infinite} \,\big] \ge 1 - \epsilon.
\end{equation}
\end{corollary}

The proof of the proposition relies on the following two lemmas.

\begin{lemma} \label{l:coex-moy1} Let $M > 0$ be such that $\P[\tau \le M] > p_c$. Let $\epsilon > 0$. There exists $A > 0$ such that, for all $x \in \Z^d$ satisfying $\|x\|_1 \ge A$ and for all $N \ge 1$, we have
\[
\frac{1}{N} \#\Big\{ k \in \{1, \dots, N\} : \E[B^M_{kx}(0,x)] \ge \mu(x)(1 - \epsilon) \Big\} \ge 1 - \epsilon.
\]
\end{lemma}

\begin{proof}
Let $\epsilon > 0$. Fix $A > 0$ such that, for all $x$ satisfying $\|x\|_1 \ge A$, we have 
\[
\mu(x)(1 - \epsilon) \le \E[T^M(0,x)] \le \mu(x)(1 + \epsilon).
\]
Let $x$ satisfy $\|x\|_1 \ge A$.
For all $k \ge 1$, we have by \eqref{e:B-triangle},
\[
\E[B^M_{kx}(0,x)] \le \E[T^M(0,x)] \le \mu(x)(1 + \epsilon)
\]
and thus 
\begin{equation}\label{e:pos1}
\mu(x)(1 + \epsilon) - \E[B^M_{kx}(0,x)] \ge 0.
\end{equation}
Moreover, for all $N \ge 1$, we have 
\[
\frac{1}{N} \sum_{k=1}^{N} \E[B^M_{kx}(0,x)] = \frac{1}{N} \E\left[ T^M(0,Nx) \right] \ge \frac{1}{N} \mu(N x)(1 - \epsilon) = \mu(x)(1 - \epsilon)
\]
and thus
\begin{equation} \label{e:moy1}
\frac{1}{N} \sum_{k=1}^{N} \Big( \mu(x)(1 + \epsilon) - \E[B^M_{kx}(0,x)] \Big) \le 2 \mu(x) \epsilon.
\end{equation}
From \eqref{e:pos1} and \eqref{e:moy1}, we deduce
\[
\frac{1}{N} \#\Big\{ k \in \{1, \dots, N\} : \mu(x)(1 + \epsilon) - \E[B^M_{kx}(0,x)] \ge \mu(x) \sqrt{2\epsilon} \Big\} \le \sqrt{2\epsilon}
\]
and thus
\[
\frac{1}{N} \#\Big\{ k \in \{1, \dots, N\} : \E[B^M_{kx}(0,x)] \ge \mu(x)(1 - \sqrt{2\epsilon}) \Big\} \ge 1 - \sqrt{2\epsilon}.
\]
\end{proof}

\begin{lemma}  \label{l:coex-moy2} Let $M > 0$ be such that $\P[\tau \le M] > p_c$. 
Let $\epsilon > 0$. There exists $A > 0$ such that, for all $x \in \Z^d$ satisfying $\|x\|_1 \ge A$, we have, for all $k \ge 1$,
\[
\E[B^M_{kx}(0,x)] \ge \mu(x)(1 - \epsilon) \implies \P[B^M_{kx}(0,x) \ge \mu(x)(1 - 2\sqrt{\epsilon})] \ge 1 - 2\sqrt{\epsilon}.
\]
\end{lemma}
\begin{proof} 
Fix $A > 0$ such that, for all $x$ satisfying $\|x\|_1 \ge A$, we have 
\[
\P[T^M(0,x) \ge \mu(x)(1 + \epsilon)] \le \epsilon
\]
and
\[
\E[|T^M(0,x) - \mu(x)|] \le \epsilon \mu(x).
\]
Let $x$ satisfy $\|x\|_1 \ge A$.
Let $k \ge 1$ such that
\[
\E[B^M_{kx}(0,x)] \ge \mu(x)(1 - \epsilon).
\]
Consider the good event 
\[
G = \{T^M(0,x) \le \mu(x)(1 + \epsilon) \}.
\]
We have 
\begin{align*}
\E[B^M_{kx}(0,x)\1_G] 
& = \E[B^M_{kx}(0,x)] - \E[B^M_{kx}(0,x)\1_{G^c}] \\
& \ge \mu(x)(1 - \epsilon) - \E[T^M(0,x)\1_{G^c}] \\
& \ge \mu(x)(1 - \epsilon) - \E[|T^M(0,x) - \mu(x)|] - \mu(x)\P[G^c] \\
& \ge \mu(x)(1 - 3\epsilon)
\end{align*}
and thus
\begin{equation}\label{e:moy2}
\E[\mu(x)(1 + \epsilon) - B^M_{kx}(0,x)\1_G] \le 4\epsilon \mu(x).
\end{equation}
But 
\begin{equation}\label{e:pos2}
\mu(x)(1 + \epsilon) - B^M_{kx}(0,x)\1_G \ge 0.
\end{equation}
From \eqref{e:moy2} and \eqref{e:pos2}, we deduce
\[
\P[\mu(x)(1 + \epsilon) - B^M_{kx}(0,x)\1_G \ge 2\sqrt{\epsilon} \mu(x)] \le 2\sqrt{\epsilon}.
\]
Hence,
\[
\P[B^M_{kx}(0,x)\1_G \ge \mu(x)(1 - 2\sqrt{\epsilon})] \ge 1 - 2\sqrt{\epsilon}
\]
then
\[
\P[B^M_{kx}(0,x) \ge \mu(x)(1 - 2\sqrt{\epsilon})] \ge 1 - 2\sqrt{\epsilon}.
\]
\end{proof}

\begin{proof}[Proof of Proposition \ref{p:coex}]
The first part follows from Lemmas \ref{l:coex-moy1} and \ref{l:coex-moy2} (applied with $\epsilon^2 / 4$).
The second part follows from the first.
Indeed, let $G_k$ be the event $\{B^M_{kx}(0,x) \ge \mu(x)(1 - \epsilon)\}$. We have $\P[\limsup_k G_k] = \lim_\ell \P[\cup_{k \ge \ell} G_k] \ge 1 - \epsilon$
because $\P[G_k] \ge 1 - \epsilon$ for infinitely many $k$.
\end{proof}

\begin{proof}[Proof of Corollary \ref{c:coex}] Let $\epsilon > 0$. Fix $M > 0$ such that $\P[\tau \le M] > p_c$ and $\P[0 \in \cC^M] \ge 1 - \epsilon$.
If $k$ and $x$ are such that $\P[B^M_{kx}(0,x) \ge \mu(x)(1 - \epsilon)] \ge 1 - \epsilon$, then we have
\begin{align*}
\P[B_{kx}(0,x) \ge \mu(x)(1 - \epsilon)] 
& \ge \P[0^M = 0 \text{ and } x^M = x \text{ and } (kx)^M = kx \text{ and } B^M_{kx}(0,x) \ge \mu(x)(1 - \epsilon)] \\
& \ge \P[B^M_{kx}(0,x) \ge \mu(x)(1 - \epsilon)] - 3\epsilon \\
& \ge 1 - 4\epsilon.
\end{align*}
The first part of the corollary follows from the first part of Proposition \ref{p:coex}.
The second part of the corollary follows from the second.
\end{proof}

\subsection{Proof of Theorem \ref{t:coex-general}}

\begin{proof}[Proof of Theorem \ref{t:coex-general}]

Let $A$ be the constant given by Corollary \ref{c:coex} for $\epsilon=1/4$. 
Given that we are increasing $A$, we may assume $A$ is an integer.
Set
\[
\overline L^+ = \limsup_{k \to \infty} B_{ke_1}(0,Ae_1) \text{ and } \underline L^- = \liminf_{k \to -\infty} B_{ke_1}(0,Ae_1).
\]
These quantities are finite thanks to \eqref{e:B-triangle}.
Consider the event
\[
G = \left\{ \overline L^+ - \underline L^- > 0\right\}.
\]
Applying \eqref{e:c:coex:infini} to $x=Ae_1$ we get the first part of the next display.
Applying \eqref{e:c:coex:infini} to $x=-Ae_1$, using antisymmetry of $(x,y) \to B_z(x,y)$ and stationarity of the model, we get the second part of the display.
To sum up:
\[
\P\left[\overline L^+ \ge \frac 3 4 \mu(Ae_1)\right] \ge \frac 3 4 \text{ and } \P\left[\underline L^- \le - \frac 3 4 \mu(Ae_1)\right] \ge \frac 3 4.
\]
As $\mu$ is a norm we get $\mu(Ae_1)>0$ and thus
\[
\P[G] \ge \P\left[\overline L^+ \ge \frac 3 4 \mu(Ae_1) \text{ and } \underline L^- \le - \frac 3 4 \mu(Ae_1) \right] \ge \frac 1 2.
\]
For all $a \in \{0,\dots,A-1\}$ we set
\[
\overline\ell^+(a) = \limsup_{k \to \infty} B_{ke_1}(ae_1,(a+1)e_1) \text{ and } \underline\ell^-(a) = \liminf_{k \to -\infty} B_{ke_1}(ae_1,(a+1)e_1) 
\]
and define
\[
G(a) = \{\underline \ell^-(a) < \overline \ell^+(a)\}.
\]
Stationarity ensures that  $\P[G(a)]$ does not depend on $a$.
On $G$ we have:
\[
\max_{0 \le a \le A-1} \big( \overline\ell^+(a) - \underline\ell^-(a) \big)
  \ge \frac 1 A \sum_{a=0}^{A-1} \big( \overline\ell^+(a) - \underline\ell^-(a) \big) 
  \ge \frac 1 A \big( \overline L^+ - \underline L^- \big) 
  > 0
\]
We uses \eqref{e:B-basic} for the second inequality and the assumption "$G$ occurs" for the third one.
We then have
\[
G \subset \bigcup_{a=0}^{A-1} G(a).
\]
As $\P[G(a)]$ does not depend on $a$ we deduce, by union bound, the inequality
\[
\P[G(0)] \ge \frac 1 A \P[G] > 0.
\]
This ends the proof.
\end{proof}

\bibliographystyle{plain}

\end{document}